\documentclass[11pt,leqno]{article}

\usepackage{amsthm,amsfonts,amssymb,amsmath,oldgerm,color,bm,multicol}

\usepackage{amsthm,amsfonts,amssymb,amsmath,color}

\numberwithin{equation}{section}

\renewcommand\d{\partial}

\renewcommand\o{\omega}
\newcommand\s{\sigma}

\newcommand\R{\mathbb R}\newcommand\N{\mathbb N}\newcommand\Z{\mathbb Z}

\def\de{\delta}

\def\OO{\Omega}

\def\epsilon{\varepsilon}
\def\e{\varepsilon}

\def\eps{\varepsilon}

\newcommand\br{\begin{rem}}
\newcommand\er{\end{rem}}
\newcommand\bp{\begin{pmatrix}}
\newcommand\ep{\end{pmatrix}}
\newcommand\be{\begin{equation}}
\newcommand\ee{\end{equation}}
\newcommand\ba{\begin{equation}\begin{aligned}}
\newcommand\ea{\end{aligned}\end{equation}}

\newcommand\nn{\nonumber}

\usepackage{layout}

\newcommand{\calF}{\mathcal{F}}

\newcommand{\supp}{{\rm supp }}

\newcommand{\uu}{{\mathbf u}}
\newcommand{\vv}{{\mathbf v}}

\newcommand{\vu}{\vc{u}}

\newcommand{\vg}{\vc{g}}
\newcommand{\vc}[1]{{\bf #1}}

\newcommand{\dx}{{\rm d} {x}}

\newcommand{\dive}{{\rm div\,}}

\newtheorem{defi}{Definition}[section]
\newtheorem{theorem}[defi]{Theorem}
\newtheorem{proposition}[defi]{Proposition}
\newtheorem{lemma}[defi]{Lemma}

\newtheorem{remark}[defi]{Remark}
\newtheorem{ass}[defi]{Assumption}

\numberwithin{equation}{section}

\begin{document}

\title{Homogenization of Poisson and Stokes equations in the whole space}

\author{Yong Lu \footnote{Department of Mathematics, Nanjing University, 22 Hankou Road, Gulou District, 210093 Nanjing, China. Email: luyong@nju.edu.cn. The work of the author has been supported by the Recruitment Program of Global Experts of China. This work is partially supported by project ANR JCJC BORDS funded by l'ANR of France.} }

\date{}

\maketitle

\begin{abstract}

We consider the homogenization of the Poisson and the Stokes equations in the whole space perforated with periodically distributed small holes. The periodic homogenization in bounded domains is well understood, following the classical results in \cite{Tartar80, Cioranescu-Murat, All-NS1,All-NS2}. In this paper, we show that these classical homogenization results in a bounded domain can be extended to the whole space $\R^d$. Our results cover all three cases corresponding to different sizes of holes and cover all $d\geq 2$.

\end{abstract}






\section{Introduction}

The homogenization of the Poisson and the Stokes equations in a bounded domain perforated with a large number of small holes has been systematically studied in many literatures following the classical papers \cite{Cioranescu-Murat} for the Poisson equation and \cite{Tartar80, All-NS1, All-NS2} for the Stokes equations.

The perforated domain under consideration is described as following. Let $\OO\subset \R^d, \ d\geq 2$ be an open domain and let $\e, a_\e$ be small parameters satisfying $0< a_{\e} \leq \e \leq 1.$ The holes are denoted by $T_{\e,k}$ which are assumed to satisfy
\be\label{def-holes1}
B(\e x_k, \de_1 a_{\e}) \subset \subset T_{\e, k}  = \e x_k + a_{\e} T  \subset  \subset B(\e x_k, \de_2\e)   \subset \subset \e Q_k,
\ee
where the cube $Q_k : = (-\frac{1}{2},\frac{1}{2})^d + k$ and $x_k = x_0 + k$ with $x_0 \in Q_0 =  (-\frac{1}{2},\frac{1}{2})^d$, for each $k\in \Z^d$; $T$ is the {\em model hole} which is assumed to be closed, bounded, and simply connected, with $C^{1}$ boundary; $\de_{i}, \, i=1,2$ are fixed positive numbers.  The perforation parameters $\e$ and $a_{\e}$ are used to measure the mutual distance of the holes and the size of the holes, respectively,  and $\e x_{k} = \e x_{0} + \e k$ determine the locations of the holes. The perforated domain $\OO_\e$ is then defined as
\be\label{def-Oe}
\OO_\eps:=\OO \setminus  \bigcup_{k\in \Z^d} T_{\e,k}. \nn
\ee

In this paper, we consider the following Dirichlet problems of the Poisson and the Stokes equations in $\OO_\e$:
\be\label{Poisson-Oe}
\left\{\begin{aligned}
-\Delta u_\e &= f,\quad &&\mbox{in}~\OO_\e,\\
u_\e & = 0,\quad &&\mbox{on} ~\d\OO_\e,
\end{aligned}\right.
\ee
\be\label{Stokes-Oe}
\left\{\begin{aligned}
-\Delta \vv_\e +\nabla p_\e &= \vg,\quad &&\mbox{in}~\OO_\e,\\
\dive \vv_\e &=0 ,\quad &&\mbox{in}~\OO_\e,\\
\vv_\e & = 0,\quad &&\mbox{on} ~\d\OO_\e.
\end{aligned}\right.
\ee

\medskip

Cioranescu and Murat \cite{Cioranescu-Murat} considered \eqref{Poisson-Oe} and Allaire \cite{All-NS1, All-NS2} considered \eqref{Stokes-Oe}, where $\OO$ is assumed to be a bounded domain with smooth boundary (for example a bounded $C^1$ domain), and the external forces $f\in L^2(\OO),  \ \vg\in L^2(\OO;\R^d)$. In their studies, instead of giving specific assumptions on the holes configurations as in \eqref{def-holes1}, some abstract framework of hypotheses is imposed. It was shown that when the number of holes goes to infinity and the size of the holes goes to zero simultaneously, the solution approaches an effective state governed by certain homogenized equations which are defined in homogeneous domains --- domains without holes. The homogenized equations are crucially determined by the ratio between the size of the holes and the mutual distance between the holes. Precisely, Allaire \cite{All-NS1, All-NS2} showed that the homogenized equations for \eqref{Stokes-Oe} are determined by the ratio $\s_{\e}$ defined as following:
\be\label{ratio}
\s_{\e}: =  \big(\frac{\e^{d}}{a_{\e}^{d-2}}\big)^{\frac{1}{2}} \quad  \mbox{if} \  d\geq 3;\qquad
\s_{\e} : =  \e\left|\log \frac{a_{\e}}{\e}\right|^{\frac 12} \quad   \mbox{if} \  d = 2.
\ee
If $\lim_{\e\to 0} \s_{\e} = 0$ corresponding to the supercritical case of {\em large holes}, the homogenized system is the Darcy's law; if $\lim_{\e\to 0} \s_{\e}  = \infty$ corresponding to the subcritical case of {\em small holes}, the limit system remains to be the same Stokes equations; if $\lim_{\e\to 0} \s_{\e}  = \s_{*} \in (0,+\infty)$ corresponding to the case of \emph{critical size of holes}, the homogenized equations are governed by the Stokes-Brinkman equations --- a combination of the Darcy's law and the original Stokes equations. 

\medskip

The homogenization studies in \cite{Cioranescu-Murat} and \cite{All-NS1,All-NS2} are extended in different perspectives. More complicated models in fluid mechanics are considered, see \cite{Mik}, \cite{Mas-Hom}, \cite{FNT-Hom}, \cite{FL1, DFL, Lu-Schwarz18}, and the references therein. By employing the idea of cell problem introduced by Tartar \cite{Tartar80}, a new unified proof covering different sizes of holes is given in \cite{Jin19} and \cite{Lu-Unify} for the homogenization of the Poisson equation and the Stokes system, respectively. Another direction of research is to consider more general holes configurations without periodicity, see \cite{DGR2008, FeNaNe, Hillairet1,CH18, GHV18, GH20}. Without periodicity, Hillairet \cite{Hillairet1} considered the Stokes problem with nonzero boundary values on the holes, where the modelling goes back to \cite{DGR2008}. In \cite{DGR2008} and \cite{Hillairet1}, the minimal distance between the holes is assume to be much larger than the size of the holes. Very recently in \cite{CH18, GHV18, GH20}, the random homogenization of the Poisson equation and the Stokes equations is studied.  Particularly in \cite{GHV18, GH20}, randomly distributed spherical holes are considered, where the centers of the holes are distributed according to a Poisson point process. They imposed very week assumptions on the holes configurations, where the holes are allowed to be very close or even overlap. For the random homogenization study in \cite{CH18}, the overlap of holes is negligible in probability. In \cite{FeNaNe, Hillairet1,CH18, GHV18, GH20}, the critical size of holes is considered in a bounded domain in $\R^d$ with $d\geq 3$ and the Brinkman type equations are derived.

\medskip

So far, most of the results are obtained for bounded domains. Recently in \cite{HV2018}, along with others, H\"ofer and Velazquez studied the homogenization of the Poisson and the Stokes problems in the whole space in three and higher dimensional setting. They employed the reflection method and derived the Brinkman type equations in the limit. In \cite{HV2018},  a new abstract framework in functional analysis was built to describe the homogenization problems, and show the connection between the method of reflections and such abstract framework. Unlike the holes configurations \eqref{def-holes1} considered in this paper, no periodicity is assumed in \cite{HV2018}. Instead, some general assumptions are imposed on the holes, such as the size, the minimal distance, the upper bound of the total capacity, the convergence and the lower bound of the average capacity. See Conditions 1.1 and 1.2, Assumption 1.7 in \cite{HV2018}. The analysis in \cite{HV2018} relies on the notation of {\em screening length}, which goes back to \cite{NO01, NV06}. In \cite{NV04, NV06}, using the {\em screening estimate}, the homogenization in unbounded domains for the Poisson equation is also studied.

In \cite{HV2018}, only the critical case $\lim_{\e\to 0} \s_{\e}  = \s_{*} \in (0,+\infty)$ is included, and it seems their method is not compatible to the other two noncritical cases. Indeed, Condition 1.1 in \cite{HV2018} which ensures the boundedness of the total capacity of the holes is not satisfied for the supercritical case of large holes; Assumption 1.7  in \cite{HV2018} which ensures the positivity of  the lower bound of the average capacity is satisfied only if the ratio $\s_\e$ is bounded as $\e\to 0$. This is not satisfied for the subcritical case of small holes. The analysis in \cite{HV2018} does not work for two dimensional case. One main issue is that in two dimensions, the decay of the Green function for the Laplace or the Stokes operator is weaker and this makes it more difficult to bound the interaction between holes. Such bounds are needed for the method of reflections used in \cite{HV2018}. Another issue is on the characterization of the homogeneous Sobolev spaces $D_0^{1,2}(\R^2)$ when $d=2$. See Remarks \ref{rem-2d-1}. In this paper, we use different method and we will cover all the three cases (critical, supercritical, and subcritical) and all $d\geq 2$.

\section{Main results}

In this section, we state our main homogenization results in the whole space $\OO= \R^d$.  We first introduce some function spaces and related properties.

\subsection{Some function spaces}\label{sec:homo-sobolev}

Let $E$ be a locally Lipschitz domain in $\R^d$. For any $1\leq q\leq \infty$ and $m\in \Z_+$,  $W^{m,q}(E)$ denotes the classical Sobolev space, and $W_0^{m,q}(E)$ denotes the completion of $C_c^\infty(E)$ in $W^{m,q}(E)$. Here $C_c^\infty(E)$ is the space of smooth functions with compact support. We use $W^{-1,q}(E)$ to denote the dual space of $W_0^{1,q}(E)$. For $1\leq q <\infty$, $W^{1,q}(\R^d) = W^{1,q}_0(\R^d)$.

 If the functions are vector valued in $\R^n$, we use the notations $W^{m,q}(E;\R^n),  \ W^{m,q}_0(E;\R^n)$, $\ C_c^\infty(E;\R^n)$, and so on. Let $C_{c,div}^\infty(E;\R^d)$ be the space of divergence free functions in $C_c^\infty(E;\R^d)$. We use $\left<\cdot,\cdot\right>_{X',X}$ to denote the dual pair between a Banach space $X$ and its dual space $X'$. We often omit the subscript and simply write $\left<\cdot,\cdot\right>$ if it is clear from the context.

We now recall some concepts of the homogeneous Sobolev spaces. The materials are mainly taken from Chapter II.6 and II.7 of Galdi's book \cite{Galdi-book}. Let $1\leq q < \infty$. We define the linear space
\be\label{def-D1q}
D^{1,q}(E)=\{u\in L_{loc}^1(E)\,:\, \|\nabla u\|_{L^q(E)} <\infty \},\quad | u |_{D^{1,q}(E)}:=\|\nabla u\|_{L^q(E)}.
\ee
The space $D^{1,q}$ is generally not a Banach space. After introducing the equivalent classes
$$
[u] = \{u +c,\ c \in \R \ \mbox{is a constant}\},\quad \mbox{for any $u\in D^{1,q}(E)$},
$$
the space $\dot D^{1,q}(E)$ of all equivalence classes $[u]$  equipped with the norm
$$
\left\|[u]\right\|_{\dot D^{1,q}(E)} := |u|_{D^{1,q}(E)} =\|\nabla u\|_{L^q(E)}
$$
is a Banach space.

The semi-norm $|\cdot |_{D^{1,q}(E)}$ introduced in \eqref{def-D1q} defines a norm in $C_c^\infty(E)$. We introduce the Banach space $D_0^{1,q}(E)$ which is the completion of $C_c^\infty(E)$ with respect to the norm $|\cdot |_{D^{1,q}(E)}$. We denote by $D^{-1,q}(E)$ the dual space of $D_0^{1,q}(E)$.

For any open set $E$ in $\R^d$, there holds the following Gagliardo-Nirenberg-Sobolev inequality:  for each $1\leq q <d$, there exists a constant $C$ depending only on $q$ and $d$ such that for all $u\in C_c^{\infty}(E)$, there holds
\be\label{GNS-ineq}
\| u \|_{L^{q^*}(E)} \leq C(q,d) \| \nabla u \|_{L^q(E)}, \ \  \frac{1}{q^*} = \frac{1}{p} - \frac{1}{d}.
\ee
By density argument, the same inequality \eqref{GNS-ineq} holds for all $u\in D_0^{1,q}(E)$ with $1\leq q<d$. This means $D_0^{1,q}(E)$ is continuously embedded into $L^{q*}(E)$ if $1\leq q <d$. Moreover, if $1\leq q<d$, Galdi \cite[equation (II.7.14)]{Galdi-book} gave an equivalent characterization of $D_0^{1,q}(E)$:
\be\label{def-tD1q}
 D_0^{1,q}(E)=\big\{ u\in D^{1,q}(E): \ u\in L^{q*}(E)\  \mbox{such that $\psi u \in W_0^{1,q}(E)$ for any $\psi \in C_c^\infty (\R^d)$}\big\},\nn
\ee
with the equivalent norm
$$
\|\cdot \|_{D_0^{1,q}(E)}:= | \cdot |_{D^{1,q}(E)}+ \|\cdot \|_{L^{q*}(E)}.
$$

It becomes more tricky if $q\geq d$. Particularly if $q\geq d$ and $E=\R^d$, there holds (see (II.7.16) of \cite{Galdi-book})
 $$D_0^{1,q}(\R^d) =D^{1,q}(\R^d)  = \{[u] = u+c : \nabla u \in L^q(\R^d), \  \mbox{$c$ is a constant}\}.$$

\subsection{Homogenization results}

We first give our assumptions on the source functions.
\begin{ass}\label{ass-g} Let $d\geq 2$.
\begin{itemize}
\item[(i)] For the critical case $\lim_{\e\to 0} \s_{\e} = \s_{*} \in (0,+\infty)$, we assume $f\in W^{-1,2}(\R^d)$ and $\vg \in W^{-1,2}(\R^d;\R^d)$.

\item[(ii)] For the supercritical case $\lim_{\e\to 0} \s_{\e} = 0$, we assume $f\in L^2(\R^d)$ and $\vg \in L^2(\R^d;\R^d)$.

\item[(iii)] For the subcritical case $\lim_{\e\to 0} \s_{\e} = \infty$, we assume $f\in D^{-1,2}(\R^d)$ and $\vg \in D^{-1,2}(\R^d;\R^d)$.

\end{itemize}

\end{ass}

In \cite{Cioranescu-Murat}, \cite{All-NS1, All-NS2}, and many other literatures, the source functions are assumed to be in $L^2(\OO)$, which is a subspace of $W^{-1,2}(\OO)$ or $D^{-1,2}(\R^d)$ when $\OO$ is bounded. Actually this choice can be relaxed for the critical and subcritical cases, where $W^{-1,2}(\OO)$ source functions will be good. In bounded domains, the classical Poincar\'e inequality can always be applied. But we loose the uniformness of the Poincar\'e type inequality for the subcritical case in the whole space. We need better source functions in $D^{-1,2}(\R^d)$ for this case. We do not need more restrictions for the other two cases compared to the study in a bounded domain. We give a remark on  Assumption \ref{ass-g} (iii):
\begin{remark}\label{rem-2d-1} A sufficient condition for Assumption \ref{ass-g} (iii) is the following:
 \be\label{f-ass}
f \in L^{\frac{2d}{d+2}}(\R^d) \ \mbox{if $d\geq 3$}; \quad \mbox{$f\in L^2(\R^d)$, $f$ is compactly supported and $\int f = 0$ if $d=2$.}\nn
\ee

If $d\geq 3$, the number $\frac{2d}{d+2}$ is actually the Lebesgue conjugate number of the component $2^*$ from the Gagliardo-Nirenberg-Sobolev inequality \eqref{GNS-ineq}. This ensures $L^{\frac{2d}{d+2}}(\R^d)$ is continuously embedded into $D^{-1,2}(\R^d)$.

The 2d case is more tricky. In this case  $D_0^{1,2}(\R^2) = D^{1,2}(\R^2) $ the functions in which can only be defined up to an addition of some constant. To ensure $f\in D^{-1,2}(\R^2)$, necessarily $\left<f,1\right> = 0$ which is equivalent to $\int f = 0$  if $f$ is integrable.  If $f\in L^2_0(\R^d)$ (the subscript $0$ means zero average) and $f$ is compactly supported, then applying the Poincar\'e's inequality
$$
\|u - \left< u\right>_{\supp f}\|_{L^2(\supp f)} \leq C(\supp f) \|\nabla u\|_{L^2(\supp f)}
$$
implies
$$
\int_{\R^d} f u\,\dx  =  \int_{\supp f} f (u-\left< u\right>_{\supp f})\,\dx \leq  C(\supp f) \|f\|_{L^2} \|\nabla u\|_{L^2(\supp f)}.
$$
This means $f\in D^{-1,2}(\R^d)$. Here $\left< u\right>_{\supp f}$ denotes the average of $u$ in $\supp f$. We remark that the constant $C(\supp f)$ depends on the size of $\supp f$.

\end{remark}

\begin{remark} To ensure the well-posedness of the Poisson problem \eqref{Poisson-Oe} and the Stokes problem \eqref{Stokes-Oe}, weaker assumptions $f\in W^{-1,2}(\R^d), \ \vg \in W^{-1,2}(\R^d;\R^d)$ will be sufficient for all three cases and for all $d\geq 2.$ See Proposition \ref{prop:Sto1}.

\end{remark}

\medskip

 Before stating the theorems, we introduce the following convention: for each $u \in W_0^{1,2}(\OO_\e)$, we will naturally treat $u$ as a function in $W_0^{1,2}(\R^d)$ by imposing
\ba\label{def-extension-u}
 u =  0  \ \mbox{on}\ \R^d \setminus \OO_\e =  \bigcup_{k\in \Z^d} T_{\e,k}.\nn
\ea

\medskip

For the Poisson problem \eqref{Poisson-Oe}, we have the following result where the limits are taken up to possible extractions of subsequences.
\begin{theorem}\label{thm:Lap}
Let $\OO = \R^d, \ d\geq 2$. Let $f$ satisfy Assumption \ref{ass-g}. Then for each fixed $\e\in (0,1)$, the Poisson problem \eqref{Poisson-Oe} admits a unique weak solution $u_\e \in W_0^{1,2}(\OO_\e)$. Moreover, we have the following description of the limit system related to different sizes of holes:
\begin{itemize}

\item[(i)] If $\lim_{\e\to 0} \s_{\e} = 0$ corresponding to the case of large holes, we have
$$
{\s_{\e}^{-2}}{u_{\e}} \to u \ \mbox{weakly in} \ L^{2}(\R^d),
$$
where $u$ satisfies
\be\label{Lap-Darcy}
u = \bar w f. \nn
\ee

\item[(ii)] If $\lim_{\e\to 0} \s_{\e} = \s_{*} \in (0,+\infty)$ corresponding to the case of critical size of holes, we have
$$
 u_{\e} \to u \ \mbox{weakly in} \ W^{1,2}_{0}(\R^{d}),
$$
 where $u$ solves the Laplace-Brinkman equation:
 \be\label{Lap-Brinkman}
-\Delta u +  \s_{*}^{-2} \bar w^{-1} u = f,\quad \mbox{in}~\R^d. \nn
\ee

\item[(iii)] If $\lim_{\e\to 0} \s_{\e} = \infty$ corresponding to the case of small holes, we have
$$
u_{\e} \to u \ \mbox{strongly in} \ D_0^{1,2}(\R^d),
$$
where $u$ satisfies the Poisson equation
 \be\label{Lap-Lap}
-\Delta u = f,\quad \mbox{in}~\R^d. \nn
\ee

\end{itemize}

Here in cases (i) and (ii), $\bar w$ is a positive constant solely determined by the model hole $T$ and is given in \eqref{w-q-e-lim-lap}.

\end{theorem}

\medskip

For the Stokes problem, we have the following theorem. Even with the presence of the pressure term, no additional assumption is needed. Again, the limits are taken up to possible extractions of subsequences.
\begin{theorem}\label{thm:Sto} Let $\OO = \R^d, \ d\geq 2$. Let $\vg$ satisfy Assumption \ref{ass-g}. Then the Stokes problem \eqref{Stokes-Oe} admits a unique weak solution $(\vv_\e,p_\e)\in W_0^{1,2}(\OO_\e;\R^d) \times L_{loc}^2(\OO_\e)$ where the uniqueness of $p_\e$ is defined up to modulating constants. And there exists an extension $\tilde p_\e \in L^2_{loc}(\R^d)$ of the pressure such that $\tilde p_\e = p_\e$ in $\OO_\e$. Moreover, we have the following homogenization results:
\begin{itemize}

\item[(i)] If $\lim_{\e\to 0} \s_{\e} = 0$, there holds
$$
{\s_{\e}^{-2}}{\vv_{\e}} \to \vv \ \mbox{weakly in} \ L^{2}(\R^{d};\R^{d})
$$
and $\tilde p_\e = p_\e^{(1)} + p_\e^{(2)}$ with
$$
\nabla p_\e^{(1)} \to \nabla p \  \mbox{weakly in}  \ L^2(\R^d;\R^d), \    p_\e^{(2)}\to 0 \ \mbox{strongly in} \  L^2(\R^d).
$$
Moreover, the limit $(\vv, p)\in L^2(\R^d;\R^d) \times D^{1,2}(\R^d)$ satisfies the Darcy's law:
 \be\label{Sto-Darcy}
\left\{\begin{aligned}
\vv &= A(\vg - \nabla p),\quad &&\mbox{in}~\R^d,\\
\dive \vv &=0 ,\quad &&\mbox{in}~\R^d.
\end{aligned}\right.\nn
\ee

\item[(ii)] If $\lim_{\e\to 0} \s_{\e} = \s_{*} \in (0,+\infty)$, then
$$
\vv_{\e} \to \vv \ \mbox{weakly in} \ W^{1,2}_{0}(\R^{d};\R^{d})
$$
and $\tilde p_\e = p_\e^{(1)} + p_\e^{(2)}$ with
$$
\nabla p_\e^{(1)} \to \nabla p^{(1)} \ \mbox{weakly in}  \ W^{m,2}(\R^d;\R^d) \ \mbox{for all $m\in \N$}, \quad p_\e^{(2)}\to p^{(2)} \  \mbox{weakly in} \  L^2(\R^d).
$$
Let $p = p^{(1)} + p^{(2)}$.  Then $(\vv, p)\in W^{1,2}_{0}(\R^{d};\R^{d}) \times \big(C^\infty(\R^d) + L^2(\R^d)\big)$ satisfies the Stokes-Brinkman equations:
 \be\label{Sto-Brinkman}
\left\{\begin{aligned}
-\Delta \vv +\nabla p  + \s_{*}^{-2} A^{-1} \vv &= \vg,\quad &&\mbox{in}~\R^{d},\\
\dive \vv &=0 ,\quad &&\mbox{in}~\R^{d}.
\end{aligned}\right.\nn
\ee

\item[(iii)] If $\lim_{\e\to 0} \s_{\e} = \infty$, we have
$$
 \vv_{\e} \to \vv \ \mbox{strongly in} \ D^{1,2}_{0}(\R^{d};\R^{d}),
$$
and $\tilde p_\e = p_\e^{(1)} + p_\e^{(2)}$ with
$$
\nabla p_\e^{(1)} \to 0 \  \mbox{strongly in}  \ L^2(\R^d;\R^d), \    p_\e^{(2)}\to p \ \mbox{weakly in} \  L^2(\R^d).
$$
Moreover, the limit $(\vv, p) \in D^{1,2}_{0}(\R^{d};\R^{d}) \times  L^2(\R^d)$ solves the Stokes equations:
 \be\label{Sto-Sto}
\left\{\begin{aligned}
-\Delta \vv +\nabla p &= \vg,\quad &&\mbox{in}~\R^{d},\\
\dive \vv &=0 ,\quad &&\mbox{in}~\R^{d}.
\end{aligned}\right. \nn
\ee
\end{itemize}

Here in cases (i) and (ii),  $A$ is a constant positive definite matrix determined by the model hole $T$ and given later in \eqref{A-eta-def-2}.
\end{theorem}

\medskip

Our theorems extend the pioneering results in \cite{Cioranescu-Murat, All-NS1, All-NS2} to the whole space. Unlike in a bounded domain, a lot nice properties in the perforated whole space are missing, such as Poincar\'e inequality, Bogovskii operator, and higher integrability implying lower ones. These properties are often used in the study of homogenization. Still, we derived the same uniform estimates and the same limit systems in the whole space. This is the main novelty of this paper.  We remark that our results also extend the homogenization results in \cite{HV2018} to two dimensional case in periodic holes configurations: for any $d\geq 2$ and any source term in $W^{-1,2}(\R^d)$, we derived the Brinkman type equations for the critical case.

Another novelty of this paper is the method of proof. The perforated whole space is a {\em bad domain} in the sense that it is not bounded, nor is its complementary due to the infinite holes distributed all over the space; this means it is not an exterior domain either. However, the same reason makes the domain a good one: one can benefit from the zero boundary conditions on the holes which are everywhere in $\R^d$ and obtain a Poincar\'e type inequality (see Lemma \ref{lem-poincare} given later). This is observed for bounded domains by Tartar \cite{Tartar80} for the special case where the mutual distance is comparable to the size of the holes, and is generalized by Allaire \cite{All-NS2}. The same idea applies to unbounded domains as shown in Lemma 4.5 in \cite{HV2018}. This Poincar\'e type inequality can be used to close the energy estimates for \eqref{Poisson-Oe} and \eqref{Stokes-Oe}, and to deduce the uniform estimates for the velocity field. The only issue for this Poincar\'e type inequality is that one has an unbounded estimate constant as $\e\to 0$ for the subcritical case of small holes. This is the reason that we assume the source term in $D^{-1,2}(\R^d)$ in Assumption \ref{ass-g} for the subcritical case.

\medskip

For the Stokes problem, additional difficulties arise due to the pressure term. An observation is that the restriction operator constructed in \cite{All-NS1} only relies on local properties. It turns out that it can be applied to the whole space. Then following \cite{Tartar80} and \cite{All-NS2}, the extension of the pressure can be defined by using the restriction operator through a dual pair. We introduce suitable frequency cut-off functions for different cases and deduce uniform estimates for the pressure extension. Given the desired uniform estimates, we employ a modified cell problem and use a unified approach to prove the homogenization results, as in \cite{Jin19} and \cite{Lu-Unify}.

\medskip

The rest of the paper is devoted to the proof of our theorems. We will show the proof details only for the Stokes problem. The Poisson case can be done similarly and we only give a sketch in Section \ref{sec:lap}; actually the proof is easier without the extra troubles caused by the pressure. The paper is organized as following. In Section \ref{sec:sol}, we prove a preliminary result concerning the well-posedness of the Stokes problem \eqref{Stokes-Oe} for each fixed $\e>0$. Then in Section \ref{sec:est} we deduce our desired uniform estimates. We finally derive the limit system in Section \ref{sec:limit}. 

In the sequel, we use $C$ to denote a universal positive constant independent of $\e$.

\section{Solvability of the Stokes problem}\label{sec:sol}

We shall prove the following result:
\begin{proposition}\label{prop:Sto1}
Let $\OO = \R^d, \ d\geq 2$ and let $\vg \in W^{-1,2}(\OO_\e;\R^d)$. For each fixed $\e\in (0,1)$, the Stokes problem \eqref{Stokes-Oe} admits a unique weak solution $(\vv_\e, p_\e) \in W_0^{1,2}(\OO_\e;\R^d)\times L_{loc}^2(\OO_\e)$ such that $\dive \vv_\e = 0$ and
\be\label{Sto-weakform1}
\int_{\OO_{\e}} \nabla \vv_{\e} : \nabla \varphi \,\dx  - \int_{\OO_{\e}} p_{\e} \, \dive \varphi \,\dx  = \left<\vg, \varphi\right>, \quad \mbox{for all} \  \varphi \in C_{c}^\infty(\OO_{\e};\R^d).
\ee
The uniqueness of $p_\e$ is defined up to adding constants. Moreover, there hold the estimates
\be\label{Sto-est1}
\|\nabla \vv_\e\|_{L^2(\OO_\e)} \leq C (1+\s_\e), \quad \|\vv_\e\|_{L^2(\OO_\e)} \leq C \s_\e(1+\s_\e),
\ee
where $\s_\e$ is the ratio given in \eqref{ratio}.

\end{proposition}

Here we provide a simple proof by approximating the unbounded domain $\OO_\e$ by bounded ones. The key point is the following Poincar\'e type inequality in perforated domains. Note that Proposition \ref{prop:Sto1} can also be proved by using the classical variational method together with the following Poincar\'e type inequality.
\begin{lemma}\label{lem-poincare} Let $R>1$ and the holes $T_{\e,k}$ be given in \eqref{def-holes1}. Define
$$\OO_{\e,R} : = B(0,R)\setminus \bigcup_{k\in K_{\e,R}} T_{\e,k}, \quad K_{\e,R} : = \{k\in \Z^d: \e \overline Q_k \subset B(0,R)\}.$$
Then there holds
\be\label{poincare-1}
\|u\|_{L^2(\OO_{\e,R})} \leq C \s_\e \|\nabla u\|_{L^2(\OO_{\e,R})}, \quad \mbox{for all $u\in W^{1,2}_0(\OO_{\e,R})$,}
\ee
where $\s_\e$ is given in \eqref{ratio} and $C$ is independent of $R$. The above result holds for $R=\infty$:
\be\label{poincare-2}
\|u\|_{L^2(\OO_{\e})} \leq C \s_\e \|\nabla u\|_{L^2(\OO_{\e})}, \quad \mbox{for all} \  u \in W^{1,2}_0(\OO_{\e}).
\ee

\end{lemma}

Similar results have been shown in \cite{All-NS2} (see Lemma 3.4.1 therein) for bounded domains. In \cite[Lemma 4.5]{HV2018}, the critical case in $\R^3$ was considered. Lemma \ref{lem-poincare} can be proved similarly. For the convenience of the readers, we briefly reproduce it below.

\begin{proof}[Proof of Lemma \ref{lem-poincare}] We will only prove \eqref{poincare-2}. The proof of \eqref{poincare-1} can be done similarly. We will assume $u \in C_c^\infty(\OO_{\e})$; for general $ u \in W^{1,2}_0(\OO_{\e})$, the result follows by density argument. By \eqref{def-holes1}, we observe that for each $k\in \Z^d$,
\be\label{pt-Qeta}
B(\e x_k, \de_1 a_{\e}) \subset \subset T_{\e, k}  \subset \subset B(\e x_k, \de_2 \e)  \subset \subset \e Q_k \subset \subset  B(\e x_k, 2 \e) \subset \bigcup_{|k'-k|\leq 3} \e \overline Q_{k'},
\ee
where
$$|k-k'| = |(k_1, \cdots, k_d) - (k'_1, \cdots, k'_d)| := \sum_{i=1}^d |k_i - k_i'|.$$

For each $x\in \OO_{\e}$, there exists $k\in \Z^d$ such that $x\in \e \overline Q_k$. Denote
$$r_x = |x - \e x_k|, \quad \o_x = \frac{x-\e x_k}{|x - \e x_k|}.$$
By the fact $u =0$ on $T_{\e, k} \supset B(\e x_k, \de_1 a_{\e}) $,  we have
\ba\label{ux-Qeta}
u(x)  & = u (\e x_k + r_x \o_x)  - u (\e x_k + \de_1  a_{\e} \o_x)\\
 &= \int_{\de_1 a_\e}^{r_x} \frac{\rm d}{{\rm d}s} u(\e x_k + s \o_x)\, {\rm d}s \\
 & = \int_{\de_1 a_\e}^{r_x} (\nabla u) (\e x_k + s \o_x) \cdot \o_x \, {\rm d}s.\nn
\ea
By H\"older's inequality, direct calculation gives
\ba\label{ux-Qeta-L2}
\|u\|_{L^2(\e \overline Q_k)}^2 & \leq \int_{B(\e x_k, 2 \e)\setminus B(\e x_k, \de_1 a_{\e}) } |u(x)|^2 \,\dx = \int_{\de_1 a_{\e}}^{2\e} \int_{\mathbb{S}^{d-1}} |u(\e x_k + r \o)|^2 r^{d-1} \,{\rm d} \o \,{\rm d} r \\
&\leq  \int_{\de_1 a_\e}^{2\e} \int_{\mathbb{S}^{d-1}} \left|\int_{\de_1 a_\e}^{r} (\nabla u) (\e x_k + s \o) \cdot \o \, {\rm d}s\right|^2   r^{d-1} \,{\rm d} \o \,{\rm d} r \\
& \leq \int_{\mathbb{S}^{d-1}} \int_{\de_1 a_\e}^{2 \e} r^{d-1} \left(\int_{\de_1 a_\e}^{r} s^{-d+1}\,{\rm d}s\right) \left(\int_{\de_1 a_\e }^{r} s^{d-1} |\nabla u(\e x_k + s \o)|^2 \, {\rm d} s \right)  \,{\rm d} r \,{\rm d} \o \\
&\leq \left(\int_{\de_1 a_\e}^{2 \e } r^{d-1} \left(\int_{\de_1 a_\e}^{r} s^{-d+1}\,{\rm d}s\right){\rm d} r\right) \left( \int_{\mathbb{S}^{d-1}}  \int_{\de_1 a_\e}^{2 \e} s^{d-1} |\nabla u(\e x_k + s \o)|^2 \, {\rm d} s  \,{\rm d} \o \right)\\
& \leq C \int_{\de_1 a_\e}^{2 \e } r^{d-1}{\rm d} r   \int_{\de_1 a_\e}^{2 \e} s^{-d+1}\,{\rm d}s \int_{B(\e x_k, 2 \e)\setminus B(\e x_k, \de_1 a_{\e})} |\nabla u(x)|^2 \,\dx.
\ea
We then deduce from \eqref{ux-Qeta-L2} that
\ba\label{u-Qeta-L2-f1}
&\|u\|_{L^2(\e \overline Q_k)}^2  \leq C \e^2 \log(\frac{\e}{a_\e})  \| \nabla u \|^2_{L^{2}(B(\e x_k,2\e))}, \quad  \mbox{if $d=2$},\\
&\|u\|_{L^2(\e \overline Q_k)}^2  \leq C \e^d a_\e^{-d+2}  \| \nabla u \|^2_{L^{2}(B(\e x_k,2\e))}, \quad \mbox{if $d\geq 3$}.
\ea
Recall the definition of $\s_\e$ in \eqref{ratio}. By \eqref{pt-Qeta} and \eqref{u-Qeta-L2-f1}, we obtain
\ba\label{u-Qeta-L2-f2}
\|u\|_{L^2(\e \overline Q_k)}^2  \leq C \s_\e^2  \| \nabla u \|^2_{L^{2}(B(\e x_k,2\e))} \leq C \s_\e^2 \sum_{|k'-k|\leq 3}\| \nabla u \|^2_{L^{2}(B(\e \overline Q_{k'}))}.\nn
\ea
Thus,
\ba\label{u-Qeta-L2-f3}
\|u\|_{L^{2}(\R^d)}^2 & = \sum_{k\in \Z^d} \|u\|_{L^2(\e \overline Q_k)}^2  \\
& \leq C \s_\e^2  \sum_{k\in \Z^d} \sum_{|k'-k|\leq 3}\| \nabla u \|^2_{L^{2}(B(\e \overline Q_{k'}))}\\
 & \leq C \s_\e^2  \sum_{k'\in \Z^d} \| \nabla u \|^2_{L^{2}(B(\e \overline Q_{k'}))} \sum_{|k'|\leq 3} 1 \\
 & \leq C d^{4} \s_\e^2  \|\nabla u\|_{L^{2}(\R^d)}^2.\nn
\ea
This is our desired inequality \eqref{poincare-2}.

\end{proof}

\begin{proof}[Proof of Proposition \ref{prop:Sto1}]

Now we consider the following Stokes equations in bounded domain $\OO_{\e,R}$ with $R>1$:
\be\label{Stokes-OeR}
\left\{\begin{aligned}
-\Delta \vv_{\e,R} +\nabla p_{\e,R} &= \vg,\quad &&\mbox{in}~\OO_{\e,R},\\
\dive \vv_{\e,R} &=0 ,\quad &&\mbox{in}~\OO_{\e,R},\\
\vv_{\e,R} & = 0,\quad &&\mbox{on} ~\d\OO_{\e,R}.
\end{aligned}\right.
\ee
For the bounded $C^1$ domain $\OO_{\e, R}$, the classical theory (see \cite{Galdi-book} for instance) ensures the existence and uniqueness of weak solution $(\vv_{\e,R}, p_{\e,R}) \subset W^{1,2}_0(\OO_{\e,R};\R^d) \times L_0^{2}(\OO_{\e,R})$ in the classical weak sense: $\dive \vv_{\e,R} = 0$ and
\be\label{Sto-weakform-eR}
\int_{\OO_{\e,R}} \nabla \vv_{\e,R} : \nabla \varphi \,\dx = \left<\vg, \varphi\right>, \quad \mbox{for all} \  \varphi \in C_{c,div}^\infty(\OO_{\e,R};\R^d).
\ee
Here $L_0^{2}(\OO_{\e,R})$ denotes the space of $L^{2}(\OO_{\e,R})$ functions which are of zero average. By density argument, the test functions in \eqref{Sto-weakform-eR} can be taken as arbitrary divergence free functions in $W^{1,2}_0(\OO_{\e,R};\R^d)$. Choosing the solution $\vv_{\e,R}$ itself as a test function in \eqref{Sto-weakform-eR} implies
\ba\label{est-veR-1}
\|\nabla \vv_{\e,R} \|_{L^2(\OO_{\e,R})}^2 & \leq C \|\vg\|_{W^{-1,2}(\OO_{\e,R})}  \|\vv_{\e,R}\|_{W^{1,2}(\OO_{\e,R})} \\
& \leq C \|\vg\|_{W^{-1,2}(\OO_{\e,R})}  (\|\vv_{\e,R}\|_{L^{2}(\OO_{\e,R})} + \|\nabla \vv_{\e,R} \|_{L^2(\OO_{\e,R})}).
\ea
Applying Lemma \ref{lem-poincare} in \eqref{est-veR-1} gives
\ba\label{est-veR-2}
\|\nabla \vv_{\e,R} \|_{L^2(\OO_{\e,R})}^2  \leq C \|\vg\|_{W^{-1,2}(\OO_{\e})} (1+\s_\e) \|\nabla \vv_{\e,R}\|_{L^{2}(\OO_{\e,R})}.\nn
\ea
Thus,
\ba\label{est-veR-3}
\|\nabla \vv_{\e,R} \|_{L^2(\OO_{\e,R})}  \leq C \|\vg\|_{W^{-1,2}(\OO_{\e})}(1+\s_\e), \ \| \vv_{\e,R} \|_{L^2(\OO_{\e,R})}  \leq C \|\vg\|_{W^{-1,2}(\OO_{\e})}\s_\e(1+\s_\e).
\ea

We can extend $\vv_{\e,R}$ by zero on $B(0,R)^c$  and obtain a bounded family $\{\vv_{\e,R}\}_{R>1}$ in $W^{1,2}(\OO_{\e};\R^d)$. Then up to a subsequence,
\ba\label{cov-veR}
\vv_{\e,R} \to \vv_\e  \ \mbox{weakly in $W_0^{1,2}(\OO_{\e};\R^d)$ as $R \to \infty$}.
\ea
Clearly $\dive \vv_\e = 0$ and the estimates in \eqref{Sto-est1} follow from \eqref{est-veR-3} and \eqref{cov-veR}. We shall show that $\vv_\e$ is a weak solution to the original Stokes problem \eqref{Stokes-Oe}.
Indeed, for any $\varphi \in C_{c,div}^\infty(\OO_{\e};\R^d)$, there exists $R_{\varphi}>1$ such that $\supp \,\varphi \subset \OO_{\e,R_\varphi}$. This means that $\varphi$ is a proper test function in \eqref{Sto-weakform-eR} for all $R\geq R_\varphi$. Testing \eqref{Sto-weakform-eR} by $\varphi$, passing $R\to \infty$, and applying \eqref{cov-veR} gives
\be\label{Sto-weakform-e}
\int_{\OO_\e} \nabla \vv_\e : \nabla \varphi \,\dx = \left< \vg , \varphi\right>,  \quad \mbox{for any} \  \varphi \in C_{c,div}^\infty(\OO_\e;\R^d).
\ee
For the pressure, by \eqref{Sto-weakform-e}, we may apply Lemma IV.1.1 in \cite{Galdi-book} to ensure that there exists $p_\e \in L_{loc}^2(\OO_\e)$ such that \eqref{Sto-weakform1} is satisfied.

\medskip

The uniqueness can be derived in a classical way.  Let $(\vu_\e, q_\e) \in W_0^{1,2}(\OO_\e;\R^d)\times L_{loc}^2(\OO_\e)$ be a solution to \eqref{Stokes-Oe} with $\vg = 0$, that means  $\dive \vu_\e = 0$ and
\be\label{Sto-weakform-u}
\int_{\OO_{\e}} \nabla \vu_{\e} : \nabla \varphi \,\dx  - \int_{\OO_{\e}} q_{\e} \, \dive \varphi \,\dx  = 0, \quad \mbox{for all} \  \varphi \in C_{c}^\infty(\OO_{\e};\R^d).
\ee
Choosing $\vu_\e$ as a test function in \eqref{Sto-weakform-u} implies $\|\nabla \vu_\e\|_{L^2(\OO_\e) }= 0$ and then $\vu_\e = 0$, due to $\vu_\e \in W_0^{1,2}(\OO_\e;\R^d)$. Thus,
\be\label{Sto-weakform-u-q}
 \int_{\OO_{\e}} q_{\e} \, \dive \varphi \,\dx  = 0, \quad \mbox{for all} \  \varphi \in C_{c}^\infty(\OO_{\e};\R^d).\nn
\ee
This implies $\nabla q_\e = 0$ and $q_\e$ is a constant in $\OO_\e$.  We completed the proof of Proposition \ref{prop:Sto1}.

\end{proof}

\section{Uniform estimates}\label{sec:est}

We know from Proposition \ref{prop:Sto1} that the Stokes problem \eqref{Stokes-Oe} admits a unique weak solution $(\vv_\e, p_\e)$ for merely $\vg \in W^{-1,2}(\OO_\e;\R^d)$.
However the estimates in \eqref{Sto-est1} are not enough to derive the homogenized models in Theorem \ref{thm:Sto}. In particular, we do not have any uniform estimates for the pressure so far. The goal of this section is to prove the uniform estimates given in Theorem \ref{thm:Sto} and we shall assume $\vg$ satisfy Assumption \ref{ass-g}.

\subsection{Estimates of $\vv_\e$}\label{sec:estv}

We will estimate $\vv_\e$ case by case. We first consider the critical case $\lim_{\e\to 0} \s_\e = \s_* \in (0,\infty)$. In this case $\{\s_\e\}_{0<\e<1}$ is bounded, so we can directly apply Proposition \ref{prop:Sto1} to obtain
\be\label{est-critical}
\|\vv_\e\|_{W^{1,2}(\OO_\e)} \leq C \|\vg\|_{W^{-1,2}(\R^d)}.
\ee

\medskip

We then consider the subcritical case $\lim_{\e\to 0} \s_\e = \infty$. In this case, we assume $\vg \in D^{-1,2}(\R^d;\R^d)$. By Theorem \ref{prop:Sto1},  taking $\vv_\e$ as a test function in \eqref{Sto-weakform1} gives
\ba\label{est-subcritical-1}
\|\nabla \vv_{\e} \|_{L^2(\OO_{\e})}^2  \leq C \|\vg\|_{D^{-1,2}(\OO_\e)}  \|\vv_{\e}\|_{D_0^{1,2}(\OO_{\e})} \leq C \|\vg\|_{D^{-1,2}(\R^d)}  \|\nabla \vv_{\e} \|_{L^2(\OO_{\e})}.\nn
\ea
This implies
\ba\label{est-subcritical-2}
\|\nabla \vv_{\e} \|_{L^2(\OO_{\e})} \leq C \|\vg\|_{D^{-1,2}(\R^d)}.
\ea
Unfortunately, we merely have an unbounded estimate for the $L^2$ norm of $\vv_\e$ by using the Poincar\'e inequality in Lemma \ref{lem-poincare}:
\ba\label{est-subcritical-3}
\|\vv_{\e} \|_{L^2(\OO_{\e})} \leq C \s_\e \|\nabla \vv_{\e} \|_{L^2(\OO_{\e})} \leq C \s_\e  \|\vg\|_{D^{-1,2}(\R^d)}.
\ea
However, if $d\geq 3$, since $\vv_\e \in W^{1,2}_0(\OO_\e;\R^d)$, the Gagliardo-Nirenberg-Sobolev inequality gives
\ba\label{est-subcritical-4}
\|\vv_{\e} \|_{L^{\frac{2d}{d-2}}(\OO_{\e})} \leq C \|\nabla \vv_{\e} \|_{L^2(\OO_{\e})} \leq C  \|\vg\|_{D^{-1,2}(\R^d)}.
\ea

\medskip

We finally consider the supercritical case  $\lim_{\e\to 0} \s_\e = 0$ where we assume $\vg\in L^2(\R^d;\R^d)$. Then taking $\vv_\e$ as a test function in \eqref{Sto-weakform1} and using the Poincar\'e inequality in Lemma \ref{lem-poincare} gives
\ba\label{est-supercritical-1}
\|\nabla \vv_{\e} \|_{L^2(\OO_{\e})}^2  \leq C \|\vg\|_{L^{2}(\OO_\e)}  \|\vv_{\e}\|_{L^{2}(\OO_{\e})} \leq C \s_\e \|\vg\|_{L^{2}(\R^d)}  \|\nabla \vv_{\e} \|_{L^2(\OO_{\e})}.\nn
\ea
Hence,
\ba\label{est-supercritical-2}
\|\nabla \vv_{\e} \|_{L^2(\OO_{\e})}  \leq C \s_\e \|\vg\|_{L^{2}(\R^d)} , \quad \| \vv_{\e} \|_{L^2(\OO_{\e})}  \leq C \s_\e^2 \|\vg\|_{L^{2}(\R^d)}.
\ea

\medskip

We summarize the above estimates in \eqref{est-critical}--\eqref{est-supercritical-2} into the following proposition where the weak limits are taken up to extracting subsequences.
\begin{proposition}\label{prop:estv}
Let $\OO=\R^d, \ d\geq 2$ and $\vg$ satisfy Assumption \ref{ass-g}. Then we have the following estimates for the solution $\vv_\e$ obtained in Proposition \ref{prop:Sto1}:
\begin{itemize}
\item[(i)] For the critical case $\lim_{\e\to 0} \s_{\e} = \s_{*} \in (0,+\infty)$, $$\|\vv_\e\|_{W^{1,2}(\OO_\e)} \leq C \|\vg\|_{W^{-1,2}(\R^d)}.$$
Hence, $\vv_\e \to \vv$ weakly in $W^{1,2}_0(\R^d;\R^d)$ with $\dive \vv = 0$.

\item[(ii)] For the subcritical case $\lim_{\e\to 0} \s_{\e} = \infty$, $$\|\nabla \vv_{\e} \|_{L^2(\OO_{\e})} + \s_\e^{-1}\|\vv_{\e} \|_{L^2(\OO_{\e})}   \leq C \|\vg\|_{D^{-1,2}(\R^d)}.$$
If $d\geq 3$, $\|\vv_{\e} \|_{L^{\frac{2d}{d-2}}(\OO_{\e})} \leq C\|\vg\|_{D^{-1,2}(\R^d)}.$ Hence, $\vv_\e \to \vv$ weakly in $D^{1,2}_0(\R^d;\R^d)$ with $\dive \vv = 0$.

\item[(iii)] For the supercritical case $\lim_{\e\to 0} \s_{\e} = 0$, $$\|\nabla \vv_{\e} \|_{L^2(\OO_{\e})}  \leq C \s_\e \|\vg\|_{L^{2}(\R^d)} , \quad \| \vv_{\e} \|_{L^2(\OO_{\e})}  \leq C \s_\e^2 \|\vg\|_{L^{2}(\R^d)}.$$
    Hence, $\s_\e^{-2}\vv_\e \to \vv$ weakly in $L^{2}(\R^d;\R^d)$ with $\dive \vv = 0$.
\end{itemize}
\end{proposition}

In above proposition, $\dive \vv = 0$ follows directly from $\dive \vv_\e = 0$.

\subsection{Estimates of the pressure}

It is more tricky to estimate the pressure. To do this, we will employ the restriction operator constructed by Allaire in \cite{All-NS1, All-NS2} (see also earlier in \cite{Tartar80} for a specific case). Firstly we observe that Allaire's construction relies essentially on analysis in the neighbourhood of each single hole. After checking the argument in Section 2.2 \cite{All-NS1}, his construction also works for unbounded domains:
\begin{proposition}\label{prop:res}
Let $\OO = \R^d, \ d\geq 2$. For any $\vu \in W_0^{1,2}(\R^d;\R^d)$, define $R_\e (\vu)$ as following:
\ba\label{def-Reu}
&R_\e(\uu) (x) :=\uu (x), &&\mbox{if}\  x \in \R^d \setminus\big( \bigcup_{k\in \Z^d} B(\e x_k, \de_2 \e)\big),\\
&R_\e(\uu) (x) :=\uu_{\e,k} (x), &&\mbox{if} \  x \in B(\e x_k,\de_2 \e)\setminus T_{\e,k},\  \mbox{for each} \ k\in \Z^d,\nn
\ea
where $\uu_{\e,k}$ solves
\be\label{def-uek}\left\{\begin{aligned}
&-\Delta \uu_{\e,k} + \nabla p_{\e,k}  =-\Delta \uu, \ &&\mbox{in}\ B(\e x_k,\de_2 \e)\setminus T_{\e,k},\\
&\dive \uu_{\e,k}=\dive \uu + \frac{1}{|B(\e x_k,\de_2 \e)\setminus T_{\e,k}|}\int_{T_{\e,k}} \dive \uu\, \dx, \ &&\mbox{in}\ B(\e x_k,\de_2 \e)\setminus T_{\e,k},\\
&\uu_{\e,k}= \uu, \ &&\mbox{on}\ \d B(\e x_k,\de_2 \e),\\
&\uu_{\e,k}=0,\ &&\mbox{on}\ \d  T_{\e,k}.
\end{aligned}\right.\nn
\ee
Then $R_\e$ defines a linear operator (named {\em restriction operator}) from $W_0^{1,2}(\R^d;\R^d)$ to $W_0^{1,2}(\OO_\e;\R^d)$ such that
\begin{itemize}
\item[(i)] $\uu \in W_0^{1,2}(\OO_\e;\R^d) \Longrightarrow R_\e (\tilde \uu)= \uu \ \mbox{in}\ \OO_\e,\ \mbox{where $\tilde \vu$ is the zero extension of $\vu$ in $\R^d$}.$
\item[(ii)] $\uu \in W_0^{1,2}(\R^d;\R^d),\  \dive \uu =0 \ \mbox{in} \ \R^d \Longrightarrow \dive R_\e (\uu) =0 \ \mbox{in} \ \OO_\e.$
\item[(iii)] For each $\uu \in W_0^{1,2}(\R^d;\R^d)$, $\|\nabla R_\e(\uu)\|_{L^2(\OO_\e)} \leq C \, \big( \|\nabla \uu\|_{L^2(\R^d)}+\s_\e^{-1} \|\uu\|_{L^2(\R^d)}\big)$, and by the Poincar\'e inequality in
Lemma \ref{lem-poincare}, there holds $\| R_\e(\uu)\|_{L^2(\OO_\e)} \leq C \, \big( \s_\e \|\nabla \uu\|_{L^2(\R^d)}+ \|\uu\|_{L^2(\R^d)}\big)$.
\end{itemize}
\end{proposition}

The proof of Proposition \ref{prop:res} can be done exactly as in \cite{All-NS1} and we omit it. We remark that a $W^{1,q}$ version of Allaire's restriction operator is shown by the author in \cite{Lu-Stokes} for $3/2<q<3$ in a bounded domain in $\R^3$.

Applying the restriction operator $R_\e$, the extension of the pressure, denoted by $\tilde p_\e$, is defined by the following formula (see the original idea of Tartar in \cite{Tartar80} in bounded domains):
\be\label{def-tp}
\left<\nabla \tilde p_\e, \varphi \right>_{W^{-1,2}(\R^d),W^{1,2}_0(\R^d)} = \left<\nabla  p_\e, R_\e(\varphi) \right>_{W^{-1,2}(\OO_\e),W^{1,2}_0(\OO_\e)}, \quad \mbox{for all} \ \varphi \in W^{1,2}_0(\R^d;\R^d),
\ee
where $p_\e$ is the pressure of the Stokes problem \eqref{Stokes-Oe}.  Note that the above formulation \eqref{def-tp} is well defined due to the three properties of $R_\e$ in Proposition \ref{prop:res}:
 \begin{itemize}
\item  Property (iii) and the estimates of $\vv_\e$ in Proposition \ref{prop:estv} ensure $\nabla \tilde p_\e \in W^{-1,2}(\R^d;\R^d)$.

\item   Property (ii) ensures the compatibility: for each $\varphi \in W^{1,2}_0(\R^d;\R^d)$ with $\dive \varphi = 0$ one has $\dive R_{\e} (\varphi)  = 0$, then one deduces naturally from \eqref{def-tp} that
$\left<\nabla \tilde p_\e, \varphi \right> = \left<\nabla  p_\e, R_\e(\varphi) \right> = 0.$

\item  Property (i) ensures $\nabla \tilde p_{\e} = \nabla p_{\e}$ in $\OO_{\e}$.

\end{itemize}

Now we are in the position to deduce the uniform estimates of $\tilde p_\e$, case by case. We will repeatedly use the estimates of $\vv_\e$ in Proposition \ref{prop:estv}.

We first consider the critical case
$\lim_{\e\to 0} \s_\e = \s_*\in (0,\infty)$ where we have $\|\vv_\e\|_{W^{1,2}(\OO_\e)} \leq C $ from Proposition \ref{prop:estv}. Then, by Property (iii) in Proposition \ref{prop:res}, using the Stokes equations
\eqref{Stokes-Oe}, we obtain for all $\varphi \in W_0^{1,2}(\R^d;\R^d)$ that
\ba\label{est-tp-1}
|\left<\nabla \tilde p_\e, \varphi \right>_{\R^d}| & = |\left< \nabla p_\e, R_\e(\varphi) \right>_{\OO_{\e}}|  =  |\left< \Delta \vv_\e + \vg, R_\e(\varphi) \right>_{\OO_{\e}}| \\
&\leq C (\|\nabla \vv_\e\|_{L^{2}(\OO_\e)} + \|\vg\|_{W^{-1,2}(\R^d)}) \|R_\e(\varphi)\|_{W^{1,2}(\OO_\e)} \\
& \leq C (\|\nabla \varphi\|_{L^2(\R^d)} + \| \varphi \|_{L^2(\R^d)}).\nn
\ea
This means the family $\{\nabla \tilde p\}_{0<\e<1}$ is bounded in $W^{-1,2}(\R^d;\R^d)$. Thus, by the characterization of Sobolev space $W^{1,2}(\R^d)$ using Fourier transforms, we have
\be\label{est-tp-0}
\|\nabla \tilde p_\e\|_{W^{-1,2}(\R^d)}^2  = C_d \int_{\R^d} |\xi|^2 (1+|\xi|^2)^{-1} |\calF[\tilde p_\e] (\xi)|^2\,{\rm d}\xi \leq C,
\ee
where $\calF[\cdot]$ denotes the Fourier transform and $C_d$ is a constant depending only on the dimension $d$. Let $\chi\in C_c^\infty(B(0,2))$ be a cutoff function such that $0\leq \chi \leq 1$ and $\chi = 1$ on $B(0,1)$. We can decompose $\tilde p_\e = p_\e^{(1)} + p_\e^{(2)}$ with
\be\label{def-tp-12}
p_\e^{(1)} : = \chi(D) \tilde p_\e, \quad p_\e^{(2)} : = (1-\chi(D))\tilde p_\e,\nn
\ee
where $\chi(D)$ is the Fourier multiplier with symbol $\chi(\xi)$. Then by \eqref{est-tp-0}, direct calculation implies
\ba\label{def-tp-3}
\|\nabla p_\e^{(1)}\|_{W^{m,2}(\R^d)}^2 & =  C_d \int_{|\xi|\leq 2} |\xi|^2 (1+|\xi|^2)^{m} |\chi(\xi)|^2 |\calF[\tilde p_\e] (\xi)|^2\,{\rm d}\xi \\
& \leq  C_d 5^{m+1} \int_{|\xi| \leq 2} |\xi|^2 (1+|\xi|^2)^{-1} |\calF[\tilde p_\e] (\xi)|^2\,{\rm d}\xi\\
& \leq C 5^m,  \ \ \mbox{for all $ m \in \N$},
\ea
and
\ba\label{def-tp-4}
 \|p_\e^{(2)}\|_{L^2(\R^d)}^2  &= C_d \int_{|\xi|\geq 1} |1-\chi(\xi)|^2 |\calF[\tilde p_\e] (\xi)|^2\,{\rm d}\xi \\
 & \leq 2 C_d \int_{|\xi| \geq 1} |\xi|^2 (1+|\xi|^2)^{-1} |\calF[\tilde p_\e] (\xi)|^2\,{\rm d}\xi \\
 & \leq C.
\ea
This implies $ \nabla p_\e^{(1)} \in \cap_{m=1}^\infty W^{m,2}(\R^d;\R^d) = \mathcal{S}(\R^d;\R^d)$ the Schwartz class.

\medskip

Now we deal with the supercritical case $\lim_{\e\to 0} \s_\e = 0$ where we have $\|\nabla \vv_\e\|_{L^{2}(\OO_\e)} \leq C \s_\e$ and $\|\vv_\e\|_{L^{2}(\OO_\e)} \leq C \s_\e^2$ from Proposition \ref{prop:estv}. Then,
by Property (iii) in Proposition \ref{prop:res}, using \eqref{Stokes-Oe} and Lemma \ref{lem-poincare}, we obtain for all $\varphi \in W_0^{1,2}(\R^d;\R^d)$ that
\ba\label{est-tp-super-1}
|\left<\nabla \tilde p_\e, \varphi \right>_{\R^d}| & = |\left< \Delta \vv_\e + \vg, R_\e(\varphi) \right>_{\OO_{\e}}| \\
&\leq C \|\nabla \vv_\e\|_{L^{2}(\OO_\e)}\|\nabla R_\e(\varphi)\|_{L^{2}(\OO_\e)} + C \|\vg\|_{L^{2}(\R^d)}\|  R_\e(\varphi)\|_{L^{2}(\OO_\e)}   \\
& \leq C \s_\e \|\nabla R_\e(\varphi)\|_{L^2(\R^d)} + C \s_\e \| \nabla R_\e(\varphi) \|_{L^2(\R^d)}\\
& \leq C (\s_\e \|\nabla \varphi\|_{L^2(\R^d)} + \| \varphi \|_{L^2(\R^d)}).
\ea
For fxied $\e>0$, consider the semiclassical Sobolev space $W^{1,2}_{\s_\e}(\R^d)$ armed with the norm
$$
\|u\|_{W^{1,2}_{\s_\e}(\R^d)}: = \big(\s_\e^2\|\nabla \varphi\|_{L^2(\R^d)}^2 + \| \varphi \|_{L^2(\R^d)}^2 \big)^{\frac{1}{2}} = C_d \big( \int_{\R^d} (1+\s_\e^2 |\xi|^2) |\calF[u] (\xi)|^2\,{\rm d}\xi\big)^{\frac{1}{2}}.
$$
Then by \eqref{est-tp-super-1}, the family $\{\nabla \tilde p_\e\}_{0<\e<1}$ is bounded in $W^{-1,2}_{\s_\e}(\R^d;\R^d) = \big( W^{1,2}_{\s_\e}(\R^d;\R^d)  \big)'$. This means
\be\label{est-tp-super-2}
\|\nabla \tilde p_\e\|_{W^{-1,2}_{\s_\e}(\R^d)}^2  = C_d \int_{\R^d} |\xi|^2 (1+ \s_\e^2 |\xi|^2)^{-1} |\calF[\tilde p_\e] (\xi)|^2\,{\rm d}\xi \leq C.
\ee
Let $\chi\in C_c^\infty(B(0,2))$ be the same cut-off function. We decompose $\tilde p_\e = p_\e^{(1)} + p_\e^{(2)}$ with
\be\label{def-tp-super-12}
p_\e^{(1)} := \chi(\s_\e D) \tilde p_\e, \quad p_\e^{(2)}: = (1-\chi(\s_\e D))\tilde p_\e.\nn
\ee
Observing
\be\label{est-tp-super-3}
(1+ \s_\e^2 |\xi|^2)^{-1} \geq  1/5 \ \mbox{for all $|\xi|\leq 2 \s_\e^{-1}$}, \quad |\xi|^2 (1+ \s_\e^2 |\xi|^2)^{-1} \geq  \s_\e^{-2}/2 \ \mbox{for all $|\xi|\geq \s_\e^{-1}$}.\nn
\ee
Then by \eqref{est-tp-super-2} we have
\ba\label{est-tp-super-4}
\|\nabla p_\e^{(1)} \|_{L^{2}(\R^d)}^2  & = C_d \int_{|\xi|\leq 2 \s_\e^{-1}} |\xi|^2 |\chi(\s_\e \xi)|^2 |\calF[\tilde p_\e] (\xi)|^2\,{\rm d}\xi \\
&  \leq  5 C_d \int_{|\xi|\leq 2 \s_\e^{-1}} |\xi|^2 (1+ \s_\e^2 |\xi|^2)^{-1} |\calF[\tilde p_\e] (\xi)|^2\,{\rm d}\xi \\
& \leq C,
\ea
and
\ba\label{est-tp-super-5}
\|p_\e^{(2)} \|_{L^2(\R^d)}^2  &= C_d \int_{|\xi|\geq  \s_\e^{-1}}  |1 - \chi(\s_\e\xi)|^2 |\calF[\tilde p_\e] (\xi)|^2\,{\rm d}\xi \\
&\leq  2 \s_\e^{2} C_d \int_{|\xi|\geq \s_\e^{-1}} |\xi|^2 (1+ \s_\e^2 |\xi|^2)^{-1} |\calF[\tilde p_\e] (\xi)|^2\,{\rm d}\xi \\
&\leq C \s_\e^2.
\ea

\medskip

For the subcritical case $\lim_{\e\to 0} \s_\e = \infty$ where we have $\|\nabla \vv_\e\|_{L^{2}(\OO_\e)} \leq C$ from Proposition \ref{prop:estv}. Then for all $\varphi \in C_c^\infty(\R^d;\R^d)$,
\ba\label{est-tp-sub-1}
|\left<\nabla \tilde p_\e, \varphi \right>_{\R^d}| & = |\left< \Delta \vv_\e + \vg, R_\e(\varphi) \right>_{\OO_{\e}}| \\
&\leq C (\|\nabla \vv_\e\|_{L^{2}(\OO_\e)} + \|\vg\|_{D^{-1,2}(\R^d)}) \|\nabla R_\e(\varphi)\|_{L^{2}(\OO_\e)} \\
& \leq C (\|\nabla \varphi\|_{L^2(\R^d)} + \s_\e^{-1}\| \varphi \|_{L^2(\R^d)}).
\ea
We may employ the analysis in the supercritical case and decompose $\tilde p_\e = p_\e^{(1)} + p_\e^{(2)}$ with
\be\label{def-tp-sub-2}
p_\e^{(1)} := \chi(\s_\e D) \tilde p_\e, \quad p_\e^{(2)} := (1-\chi(\s_\e D))\tilde p_\e\nn
\ee
and deduce from \eqref{est-tp-sub-1} that
\be\label{def-tp-sub-3}
\|\nabla p_\e^{(1)}\|_{L^{2}(\R^d)} \leq C \s_\e^{-1},  \quad \|p_\e^{(2)}\|_{L^2(\R^d)} \leq C.
\ee

We summarize the above estimates (see \eqref{def-tp-3}, \eqref{def-tp-4}, \eqref{est-tp-super-4}, \eqref{est-tp-super-5}, \eqref{def-tp-sub-3}) into the following proposition where the limits are taken up to possible extractions of subsequences.
\begin{proposition}\label{prop:pressure}
Let $\OO = \R^d, \ d\geq 2$ and $\vg$ satisfy Assumption \ref{ass-g}. Then we have the following estimates for the pressure extension $\tilde p_\e$ defined by \eqref{def-tp}:
\begin{itemize}
\item[(i)] For the critical case $\lim_{\e\to 0} \s_{\e} = \s_{*} \in (0,+\infty)$, there exists a decomposition $\tilde p_\e = p_\e^{(1)} + p_\e^{(2)}$ with
\be
\|\nabla p_\e^{(1)}\|_{W^{m,2}(\R^d)} \leq C(m)  \ \mbox{for all $m\in \N$},  \quad \|p_\e^{(2)}\|_{L^2(\R^d)} \leq C.\nn
\ee
Hence, $\nabla p_\e^{(1)} \to \nabla p^{(1)}$ weakly in $W^{m,2}(\R^d;\R^d)$ for all $m\in \N$, $p_\e^{(2)}\to p^{(2)}$ weakly in $L^2(\R^d)$.

\item[(ii)] For the supercritical case $\lim_{\e\to 0} \s_{\e} = 0$, there exists a decomposition $\tilde p_\e = p_\e^{(1)} + p_\e^{(2)}$ with
\ba
\|\nabla p_\e^{(1)} \|_{L^{2}(\R^d)}  \leq C, \quad  \|p_\e^{(2)} \|_{L^2(\R^d)}   \leq C \s_\e.\nn
\ea
Then $\nabla p_\e^{(1)} \to \nabla p$ weakly in $L^2(\R^d;\R^d)$ and $ p_\e^{(2)}\to 0$ strongly in $L^2(\R^d)$.

\item[(iii)] For the subcritical case $\lim_{\e\to 0} \s_{\e} = \infty$, there exists a decomposition $\tilde p_\e = p_\e^{(1)} + p_\e^{(2)}$ with
\ba
\|\nabla p_\e^{(1)} \|_{L^{2}(\R^d)}  \leq C \s_\e^{-1}, \quad  \|p_\e^{(2)} \|_{L^2(\R^d)}   \leq C.\nn
\ea
Then $\nabla p_\e^{(1)} \to  0$ strongly in $L^2(\R^d;\R^d)$ and $ p_\e^{(2)}\to p$ weakly in $L^2(\R^d)$.
\end{itemize}

For all the above three cases, $\{\tilde p_\e\}_{0<\e<1}$ is bounded in $L^2_{loc}(\R^d)$. Since $p_\e$ coincides with $\tilde p_\e$ in $\OO_\e$, so $\{p_\e\}_{0<\e<1}$ is bounded in $L^2_{loc}(\OO_\e)$.

\end{proposition}

The estimates and weak convergence of $\{\vv_\e\}_{0<\e<1}$ and $\{\tilde p_\e\}_{0<\e<1}$ in Theorem \ref{thm:Sto} have been shown in Propositions \ref{prop:estv} and \ref{prop:pressure}. To complete the proof of Theorem \ref{thm:Sto}, it is left to show the limit equations and the strong convergence of $\nabla \vv_\e$ in $L^2(\R^d)$ for the subcritical case. This will be done in the next section.

\section{Limit equations}\label{sec:limit}

The goal of this section is to show the limit equations and finish the proof of Theorem \ref{thm:Sto}. To achieve such a goal, a natural way is to pass $\e\to 0$ in the weak formulation of \eqref{Stokes-Oe}. In this limit passage, there is an issue on the choice of test functions. Since the homogenized system is defined in $\R^d$,  so one needs to choose test functions in $C_c^\infty(\R^d)$. However $C_c^\infty(\R^d)$ functions are not proper test functions for the Stokes problem \eqref{Stokes-Oe} in $\OO_{\e}$, for which the test functions should be chosen in $C_c^\infty(\OO_{\e})$. Hence, a proper surgery on the test functions need to be done and this surgery plays a crucial role in the study of the homogenization problems. Tartar \cite{Tartar80} and Allaire \cite{All-NS1,All-NS2} used different ideas to solve this issue. As in \cite{Lu-Unify}, we use Tartar's idea of cell problem.

\subsection{Cell problem}

We generalize Tartar's idea and consider the following \emph{modified cell problem} introduced in \cite{Lu-Unify}:
 \be\label{pb-cell}
\left\{\begin{aligned}
-\Delta w^i_\eta + \nabla q_\eta^i &=  c_{\eta}^{2} e^i,\ &&\mbox{in}\ Q_\eta := Q_0 \setminus (\eta T),\\
\dive w_\eta^i &=0 ,\ &&\mbox{in}\ Q_\eta,\\
w_\eta^i &=0,\ &&\mbox{on} ~ \eta T,\\
 (w_\eta^i,q_\eta^i)  &\  \mbox{is $Q_0$-periodic.}
\end{aligned}\right.\nn
\ee
Here $\{e^i\}_{i=1,\cdots,d}$ is the standard Euclidean coordinate of $\R^d$; $\eta = a_\e/\e \in (0,1]$, and $c_{\eta}$ is defined as
\be\label{def-c-eta}
c_{\eta} := |\log \eta|^{-\frac{1}{2}} \quad \mbox{if $d=2$}; \quad c_{\eta} := \eta^{\frac{d-2}{2}} \quad \mbox{if $d\geq 3$}.
\ee

We collect some basic facts from Section 2 in \cite{Lu-Unify}:
\ba\label{est-weta-2}
\|\nabla w_\eta^i  \|_{L^2(Q_\eta)} \leq C c_{\eta} , \quad \| w_\eta^i  \|_{L^2(Q_\eta)} \leq C, \quad  \| q_\eta^i \|_{L^2(Q_\eta)} \leq C c_{\eta}.
\ea
Then as $\e\to 0$, up to possible extractions of subsequences,
\be\label{w-eta-lim}
w_\eta^i  \to  w^i  \ \mbox{weakly in} \ W^{1,2}(Q_0),\quad w_\eta^i  \to w^i  \ \mbox{strongly in} \ L^2(Q_0), \quad c_{\eta}^{-1} q_\eta^i  \to   q^i  \ \mbox{weakly in} \ L^2(Q_{0}).
\ee
Thus,
\be\label{A-eta-def-2}
 A(\eta)_{i,j} := c_{\eta}^{-2} \int_{Q_\eta} \nabla  w_\eta^i : \nabla  w_\eta^j \,\dx =  \int_{Q_\eta}  (w_\eta^i)_j \,\dx \to   \bar w^i_j: = \int_{Q_0}  w^i_j \,\dx,
\ee
with $A := ( \bar w^i_j)_{1\leq i,j\leq d}$ a symmetric positive definite matrix.  Moreover, the main Theorem in \cite[Section 0]{Allaire91} says that $A =  M^{-1}$, where $M$ is the permeability tensor introduced in \cite{All-NS1} or \cite{Allaire91}.

\medskip

Then define
\be\label{w-q-e-def}
w^i_{\eta,\e} (\cdot) : = w^i_{\eta} \big(\frac{\cdot}{\e}\big),\ q^i_{\eta,\e} (\cdot) : = q^i_{\eta} \big(\frac{\cdot}{\e}\big) \nn
\ee
which solve
\be\label{w-q-e-pt1}
\left\{\begin{aligned}
- \Delta w^i_{\eta,\e} + \e^{-1} \nabla q^i_{\eta,\e} & = \e^{-2} c_{\eta}^{2}  e^i = \s_\e^{-2} e^i,\ &&\mbox{in}\ \e Q_0 \setminus (a_\e T),\\
\dive w^i_{\eta,\e} &=0 ,\ &&\mbox{in}\ \e Q_0 \setminus (a_\e T),\\
w^i_{\eta,\e} &=0,\ &&\mbox{on}  \ a_\e T,\\
(w^i_{\eta,\e},q^i_{\eta,\e}) &\  \mbox{is $\e Q_0$-periodic}.
\end{aligned}\right.
\ee
Here we used the fact $c_{\eta} \e^{-1} = \s_\e^{-1}.$ For each $R>1$, by \eqref{est-weta-2} and the periodicity of $(w^i_{\eta}, q^i_{\eta})$, direct calculation gives
\ba\label{w-q-e-pt2}
\|w^i_{\eta,\e}\|_{L^2(B(0,R))} \leq C(R), \quad \|q^i_{\eta,\e}\|_{L^2(B(0,R))} \leq C (R) c_{\eta}, \quad \|\nabla w^i_{\eta,\e}\|_{L^2(B(0,R))}  \leq C(R) \s_{\e}^{-1},
\ea
where the constant $C(R)$ depends only on $R$. By \eqref{w-eta-lim}, again using the periodicity of $(w^i_{\eta}, q^i_{\eta})$ gives
\be\label{w-q-e-lim}
w^i_{\eta,\e}  \to \bar w^i  \ \mbox{weakly in} \ L^{2}_{loc}(\R^d),\quad c_{\eta}^{-1}q^i_{\eta,\e}  \to \bar q^i  \ \mbox{weakly in} \ L^2_{loc}(\R^d),
\ee
as $\e\to 0$, up to extracting subsequences. Here $\bar q^i:= \int_{Q_0} q^i \,\dx$.

\subsection{Limit passages}

Clearly $w^i_{\eta,\e}$ vanishes on the holes $T_{\e,k}$ for all $k\in \Z^d$. Thus, given a scalar function $\phi \in C_c^\infty(\R^d)$, there holds $w_{\eta,\e}^i \phi \in W_{0}^{1,2}(\OO_\e;\R^d)$.  We can take $w_{\eta,\e}^i \phi$ as a test function in the weak formulation of \eqref{Stokes-Oe} and deduce
\ba\label{Sto-O-weak}
\int_{\R^d} \nabla \vv_\e : \nabla(w_{\eta,\e}^i \phi)\,\dx - \int_{\R^d} \tilde p_\e \, \dive(w_{\eta,\e}^i \phi)\,\dx  = \langle \vg , (w_{\eta,\e}^i \phi \rangle.
\ea
Since $\vv_\e$ and $w_{\eta,\e}^i$ both vanish on the holes and $\tilde p_{\e}$ coincides with $ p_{\e}$ in $\OO_{\e}$, the integrals in \eqref{Sto-O-weak} are the same if we replace $\R^d$ by $\OO_\e$ which should be the correct one in the weak formulation of \eqref{Stokes-Oe}.  By \eqref{w-q-e-pt1}, direct calculation gives
\ba\label{Sto-O-weak-v}
\int_{\R^d} \nabla \vv_\e : \nabla(w_{\eta,\e}^i \phi)\,\dx  & =  \int_{\R^d} \nabla  \vv_\e :  (w_{\eta,\e}^i \otimes \nabla \phi)\,\dx - \int_{\R^d} (\nabla \phi \otimes  \vv_\e) : \nabla w_{\eta,\e}^i \,\dx \\
& \quad + \int_{\R^d} \nabla (\phi \vv_\e) : \nabla w_{\eta,\e}^i \,\dx \\
&  =  \int_{\R^d} \nabla  \vv_\e :  (w_{\eta,\e}^i \otimes \nabla \phi ) \,\dx - \int_{\R^d} (\nabla \phi \otimes \vv_\e ): \nabla w_{\eta,\e}^i \,\dx \\
& \quad + \e^{-1} \int_{\R^d} \dive (\phi \vv_\e) \, q_{\eta,\e}^i \,\dx  + \s_\e^{-2} \int_{\R^d} (\phi \vv_\e)\cdot e^i \,\dx.
\ea
By Proposition \ref{prop:pressure} and $\dive w_{\eta, \e}^i = 0$, we have
\ba\label{Sto-O-weak-p}
 \int_{\R^d}\tilde  p_\e \, \dive(w_{\eta,\e}^i \phi)\,\dx   & =  - \int_{\R^d}\nabla  p_\e^{(1)} \cdot (w_{\eta,\e}^i \phi)\,\dx + \int_{\R^d} p_\e^{(2)}\dive(w_{\eta,\e}^i \phi) \,\dx\\
 & = - \int_{\R^d}\nabla  p_\e^{(1)} \cdot (w_{\eta,\e}^i \phi)\,\dx + \int_{\R^d} p_\e^{(2)} \nabla \phi \cdot w_{\eta,\e}^i  \,\dx.
\ea

\medskip

 We will pass $\e\to 0$ case by case in the following subsections. Propositions \ref{prop:estv} and \ref{prop:pressure} will be used multiple times. The limits are often taken up to extractions of subsequences and we will not repeat this point.

  The constant $C$ in the following argument will often depend on the size of $\supp \phi$ due to the local integrability of $w_{\eta,\e}^i$ and $q_{\eta,\e}^i$, see \eqref{w-q-e-pt2}. We may not emphasize this dependency if it is clear from the context. Anyway, once $\phi$ is given, the constant $C(\supp \phi)$ is fixed.

\subsection{Supercritical case with large holes}

For the supercritical case $\lim_{\e \to 0} \s_{\e} \to 0$, we assume $\vg \in L^{2}(\R^d;\R^d)$. In this case, we recall some estimates that we are going to use right away (see Propositions \ref{prop:estv} and \ref{prop:pressure}):
\ba\label{est-super-sum}
& \|\nabla \vv_{\e} \|_{L^2(\R^d)}  \leq C \s_\e  , \quad \| \vv_{\e} \|_{L^2(\R^d)}  \leq C \s_\e^2 , \\
&\tilde p_\e = p_\e^{(1)} + p_\e^{(2)} \ \mbox{with} \ \|\nabla p_\e^{(1)} \|_{L^{2}(\R^d)}  \leq C, \ \|p_\e^{(2)} \|_{L^2(\R^d)}   \leq C \s_\e.
\ea

We first estimate the right-hand side of \eqref{Sto-O-weak-v}. By \eqref{w-q-e-pt2}, \eqref{w-q-e-lim}, \eqref{est-super-sum},  we have
\ba\label{Sto-O-weak-v-super-1}
\left|\int_{\R^d} \nabla  \vv_\e :  (w_{\eta,\e}^i \otimes \nabla \phi) \,\dx\right| & \leq C \| \nabla \vv_\e \|_{L^2(\R^d)} \| w_{\eta,\e}^i \|_{L^2(\supp{\phi})} \leq C(\supp \phi) \s_{\e} \to 0,\\
\left|\int_{\R^d} (\nabla \phi \otimes  \vv_\e ): \nabla w_{\eta,\e}^i \,\dx \right| & \leq C \|\vv_\e \|_{L^2(\R^d)} \| \nabla w_{\eta,\e}^i \|_{L^2(\supp{\phi})} \leq C(\supp \phi) \s_{\e}^{2} \s_{\e}^{-1}  \to 0.
\ea
Here $C(\supp \phi)$ depends on the size of the compact set $\supp \phi$. Using the divergence free condition $\dive  \vv_{\e} = 0$ and observing  $ \e^{-1} c_{\eta} = \s_{\e}^{-1}$ implies
\ba\label{Sto-O-weak-v-super-2}
\left|\e^{-1} \int_{\R^d} \dive (\phi \vv_\e) \, q_{\eta,\e}^i \,\dx \right|  & = \left|\e^{-1} \int_{\R^d} \nabla \phi\cdot  \vv_\e \, q_{\eta,\e}^i \,\dx \right| \\
&\leq C \e^{-1}\| \vv_\e \|_{L^2(\R^d)} \| q_{\eta,\e}^i \|_{L^2(\supp{\phi})} \\
&\leq C \e^{-1}\s_{\e}^{2} c_{\eta}  = C \s_{\e} \to 0.
\ea
Since $\s_\e^{-2}\vv_\e \to \vv$ weakly in $L^2(\R^d;\R^d)$, then
\ba\label{Sto-O-weak-v-super-3}
 {\s_\e^{-2}} \int_{\R^d} \phi  {\vv_\e}\cdot e^i \,\dx  \to \int_{\R^d} \phi \vv \cdot e^i \,\dx.
\ea

For the terms related to the pressure in \eqref{Sto-O-weak-p}, by \eqref{est-super-sum} and \eqref{w-q-e-lim}, we have
\ba\label{Sto-O-weak-p-super}
 \int_{\R^d}\tilde  p_\e \, \dive(w_{\eta,\e}^i \phi)\,\dx  & = - \int_{\R^d}\nabla  p_\e^{(1)} \cdot (w_{\eta,\e}^i \phi)\,\dx + \int_{\R^d} p_\e^{(2)} \nabla \phi \cdot w_{\eta,\e}^i  \,\dx \\
  & \to - \int_{\R^d}\nabla  p \cdot (\bar w^i \phi)\,\dx.
\ea

For the source term,
\ba\label{Sto-O-weak-g-super}
\langle \vg , (w_{\eta,\e}^i \phi \rangle = \int_{\R^d} \vg  \cdot (w_{\eta,\e}^i \phi) \,\dx \to \int_{\R^d} \vg \cdot (\bar w^i \phi) \,\dx.
\ea

Thus, by \eqref{Sto-O-weak-v-super-1}--\eqref{Sto-O-weak-g-super},  passing $\e\to 0$ in \eqref{Sto-O-weak} implies
\ba\label{Sto-O-weak-super-1}
 \int_{\R^d} \phi \vv \cdot e^i \,\dx +  \int_{\R^d} \nabla p  \cdot (\bar w^i \phi)\,\dx = \int_{\R^d} \vg \cdot \bar w^i \phi\,\dx. \nn
 \ea
This is the Darcy's law in $\R^d$:
 \ba\label{Sto-O-weak-super-2}
 \vv = A(\vg - \nabla p), \ \mbox{in} \ \R^d,\nn
 \ea
 where $A = ( \bar w^i_j)_{1\leq i,j\leq d}$, which is the constant positive definite matrix defined in \eqref{A-eta-def-2}.

 \subsection{Critical case}
 For the critical case $\lim_{\e\to 0} \s_{\e} = \s_{*}\in (0,+\infty)$, we have (see Propositions \ref{prop:estv} and \ref{prop:pressure}, \eqref{w-q-e-pt2} and \eqref{w-q-e-lim}):
 \ba\label{est-critical-sum}
 & \|\vv_{\e} \|_{W^{1,2}_0(\R^d)}  \leq C, \\
 & \tilde p_\e = p_\e^{(1)} + p_\e^{(2)} \ \mbox{with} \ \|\nabla p_\e^{(1)} \|_{W^{2,m}(\R^d)}  \leq C \ \mbox{for all $m\in \N$}, \ \|p_\e^{(2)} \|_{L^2(\R^d)}   \leq C, \\
&  \|w^i_{\eta,\e}\|_{W^{1,2}(B(0,R))}\leq C(R) \ \mbox{for all $R>1$}.
\ea
Then by Rellich-Kondrachov compact embedding theorem, $w^i_{\eta,\e}  \to \bar w^i  \ \mbox{weakly in} \ W^{1,2}(B(0,R);\R^d)$, $w^i_{\eta,\e}  \to \bar w^i  \ \mbox{srtongly in} \ L^{2}(B(0,R);\R^d)$, and $\vv_\e \to \vv$ strongly in $L^{2}(B(0,R);\R^d)$,  for all $R>1$.  Choose $R$ large such that $\supp \phi \subset B(0,R)$. We then have for the right-hand side of \eqref{Sto-O-weak-v}:
 \ba\label{Sto-O-weak-v-critical-1}
\int_{\R^d} \nabla \vv_\e :  ( w_{\eta,\e}^i \otimes \nabla \phi ) \,\dx & \to \int_{\R^d} \nabla \vv :  (\bar w^i \otimes \nabla \phi) \,\dx = \int_{\R^d} \nabla \vv :  \nabla(\bar w^i  \phi)\,\dx,\\
\int_{\R^d} (\nabla \phi \otimes  \vv_\e ) : \nabla w_{\eta,\e}^i \,\dx & \to \int_{\R^d} (\nabla \phi \otimes  \vv ): \nabla \bar w^i \,\dx =0,\\
\e^{-1} \int_{\R^d} \dive (\phi \vv_\e) \, q_{\eta,\e}^i \,\dx & =\e^{-1} c_{\eta}\int_{\R^d} \dive (\phi \vv_\e) \,  (c_{\eta}^{-1} q_{\eta,\e}^i)  \,\dx  = \s_{\e}^{-1} \int_{\R^d} \nabla \phi \cdot \vv_\e \,  (c_{\eta}^{-1} q_{\eta,\e}^i) \,\dx \\
& \to  \s_{*}^{-1}\int_{\R^d} \nabla \phi \cdot \vv \, \bar q^i \,\dx = \s_{*}^{-1} \int_{\R^d} \dive ( \phi  \vv) \, \bar q^i \,\dx  = 0,\\
 \s_{\e}^{-2}\int_{\R^d} \phi \vv_\e \cdot e^i \,\dx  & \to \s_{*}^{-2}\int_{\R^d} \phi \vv \cdot e^i \,\dx,
\ea
where we used the fact that $\bar w^i$ and $\bar q^i$ are constant.

\medskip

Similarly, for the terms related to the pressure in \eqref{Sto-O-weak-p}, we have
 \ba\label{Sto-O-weak-p-critical-1}
 \int_{\R^d} p_\e \, \dive(w_{\eta,\e}^i \phi)\,\dx  & = -  \int_{\R^d}\nabla  p_\e^{(1)} \cdot (w_{\eta,\e}^i \phi)\,\dx + \int_{\R^d} p_\e^{(2)} \nabla \phi \cdot w_{\eta,\e}^i  \,\dx \\
 & \to  - \int_{\R^d}\nabla  p^{(1)} \cdot (\bar w^i \phi)\,\dx + \int_{\R^d} p^{(2)} \nabla \phi \cdot \bar w^i  \,\dx   = \int_{\R^d} p \dive (\bar w^i \phi)\,\dx,
\ea
where $p = p^{(1)} + p^{(2)}$.

\medskip

For the source term, since $w_{\eta,\e}^i \phi \to \bar w^i \phi$  weakly in $W^{1,2}(\R^d;\R^d)$, there holds
\be\label{Sto-O-weak-g-critical-1}
 \langle \vg, w_{\eta,\e}^i \phi \rangle  \to \langle \vg,\bar w^i \phi \rangle.
\ee

Finally, by \eqref{Sto-O-weak-v-critical-1}, \eqref{Sto-O-weak-p-critical-1} and \eqref{Sto-O-weak-g-critical-1}, passing $\e\to 0$ in \eqref{Sto-O-weak} implies
\ba\label{Sto-O-weak-critical-1}
\int_{\R^d} \nabla \vv :  \nabla(\bar w^i  \phi)\,\dx + \s_{*}^{-2}\int_{\R^d} \phi \vv \cdot e^i \,\dx - \int_{\R^d} p   \, \dive(\bar w^i \phi)\,\dx = \langle \vg,\bar w^i \phi \rangle . \nn
 \ea
This is the Brinkman type equations in the sense of distribution in $\R^d$:
 \ba\label{Sto-O-weak-critical-2}
 \s_{*}^{-2}\vv = A(\vg - \nabla p + \Delta \vv) \ \Longleftrightarrow \ - \Delta \vv + \nabla p  +  \s_{*}^{-2} A^{-1}\vv  =  \vg.\nn
 \ea

\subsection{Subcritical case with small holes}

For the subcritical case $\lim_{\e \to 0} \s_{\e} \to \infty$, we recall the estimates (see Propositions \ref{prop:estv} and \ref{prop:pressure}): 
\ba\label{est-sub-sum}
& \|\nabla \vv_{\e} \|_{L^2(\R^d)}  \leq C , \quad \| \vv_{\e} \|_{L^2(\R^d)}  \leq C \s_\e, \\
&\tilde p_\e = p_\e^{(1)} + p_\e^{(1)} \ \mbox{with} \ \|\nabla p_\e^{(1)} \|_{L^{2}(\R^d)}  \leq C \s_\e^{-1},  \  \|p_\e^{(2)} \|_{L^2(\R^d)}   \leq C.
\ea

By \eqref{w-q-e-pt2} and \eqref{w-q-e-lim}, for each $R>1$ we have $\|\nabla w^i_{\eta,\e}\|_{L^2(B(0,R))}\leq C(R) \s_{\e}^{-1} \to 0$  and $w^i_{\eta,\e}  \to \bar w^i  \ \mbox{srtongly in} \ L^{2}(B(0,R);\R^d)$.  Thus, as $\e \to 0$,
\ba\label{Sto-O-weak-v-sub}
\int_{\R^d} \nabla \vv_\e : \nabla(w_{\eta,\e}^i \phi)\,\dx  & =  \int_{\R^d} \nabla \vv_\e :  (w_{\eta,\e}^i \otimes \nabla \phi) \,\dx + \int_{\R^d} \nabla  \vv_\e : \nabla w_{\eta,\e}^i \phi\,\dx \\
& \to \int_{\R^d} \nabla  \vv : ( \bar w^i \otimes \nabla \phi ) \,\dx = \int_{\R^d} \nabla  \vv :  \nabla(\bar w^i \phi)\,\dx ,
\ea
\ba\label{Sto-O-weak-p-sub}
 \int_{\R^d}\tilde  p_\e \, \dive(w_{\eta,\e}^i \phi)\,\dx   = - \int_{\R^d}\nabla  p_\e^{(1)} \cdot (w_{\eta,\e}^i \phi)\,\dx + \int_{\R^d} p_\e^{(2)} \nabla \phi \cdot w_{\eta,\e}^i  \,\dx  \to \int_{\R^d} p   \, \dive(\bar w^i \phi)\,\dx,
\ea
and
\ba\label{Sto-O-weak-g-sub}
\langle \vg, (w_{\eta,\e}^i \phi) \rangle  \to   \langle \vg, (\bar w^i \phi) \rangle.
\ea
By \eqref{Sto-O-weak-v-sub}--\eqref{Sto-O-weak-g-sub}, passing $\e\to 0$ in \eqref{Sto-O-weak} implies
\ba\label{Sto-O-weak-sub-1}
 \int_{\R^d} \nabla \vv :  \nabla(\bar w^i \phi)\,\dx -  \int_{\R^d} p   \, \dive(\bar w^i \phi)\,\dx =  \langle \vg , \bar w^i \phi\rangle.\nn
\ea
This gives
\ba\label{Sto-O-weak-sub-2}
\int_{\R^d} \nabla \vv : \nabla (A\varphi) - p \,\dive (A\varphi)  \,\dx = \langle \vg , A\varphi\rangle, \ \forall \, \varphi \in C_c^\infty(\R^d;\R^{d}),\nn
\ea
which means
\ba\label{Sto-O-weak-sub-3}
 -\Delta \vv + \nabla p = \vg,
\ea
in the sense of distribution in $\R^d$, due to the positivity of $A$.

\medskip

At the end, we show the strong convergence $ \vv_{\e} \to  \vv$ in $D^{1,2}_{0}(\R^d;\R^{d})$. Taking $\vv_{\e}$ as a test function in the weak formulation of \eqref{Stokes-Oe}, using the weak convergence $\nabla \vv_{\e} \to \nabla \vv$ in $L^{2}(\R^d;\R^{d\times d})$, and passing $\e \to 0$ implies
\ba\label{st-conv-1}
\lim_{\e \to 0}  \| \nabla  \vv_{\e}\|_{L^{2}(\R^d)}^{2}  = \langle  \vg,\vv\rangle.\nn
\ea
Taking $\vv$ as a test function to \eqref{Sto-O-weak-sub-3} gives
\ba\label{st-conv-2}
\| \nabla \vv\|_{L^{2}(\R^d)}^{2}  = \langle  \vg,\vv\rangle.\nn
\ea
This gives $\lim_{\e \to 0}  \| \nabla  \vv_{\e}\|_{L^{2}(\R^d)}  =   \| \nabla \vv\|_{L^{2}(\R^d)}$ resulting in $\nabla \vv_{\e} \to \nabla\vv$ strong in $L^{2}(\R^d;\R^{d\times d})$. We thus complete the proof of Theorem \ref{thm:Sto}.

\section{The Poisson problem}\label{sec:lap}
The proof of the homogenization results of the Poisson problem is similar to the Stokes case but only easier.  We briefly show some steps. The following result corresponds to Propositions \ref{prop:Sto1} and \ref{prop:estv}. The limits are taken up to possible extractions of subsequences.
\begin{proposition}\label{prop:estv-lap}
Let $\OO = \R^d, \ d\geq 2$ and $f$ satisfy Assumption \ref{ass-g}. Then the Poisson problem \eqref{Poisson-Oe} admits a unique solution $u_\e \in W^{1,2}_0(\OO_\e)$ with the following estimates:
\begin{itemize}
\item[(i)] For the critical case $\lim_{\e\to 0} \s_{\e} = \s_{*} \in (0,+\infty)$, $$\|u_\e\|_{W^{1,2}(\OO_\e)} \leq C \|f\|_{W^{-1,2}(\R^d)}.$$
Hence, $u_\e \to u$ weakly in $W^{1,2}_0(\R^d)$.

\item[(ii)] For the subcritical case $\lim_{\e\to 0} \s_{\e} = \infty$, $$\|\nabla u_{\e} \|_{L^2(\OO_{\e})} + \s_\e^{-1}\|u_{\e} \|_{L^2(\OO_{\e})}   \leq C \|f\|_{D^{-1,2}(\R^d)}.$$
If $d\geq 3$, $\|u_{\e} \|_{L^{\frac{2d}{d-2}}(\OO_{\e})} \leq C\|f\|_{D^{-1,2}(\R^d)}.$ Hence, $u_\e \to u$ weakly in $D^{1,2}_0(\R^d)$.

\item[(iii)] For the supercritical case $\lim_{\e\to 0} \s_{\e} = 0$, $$\|\nabla u_{\e} \|_{L^2(\OO_{\e})}  \leq C \s_\e \|f\|_{L^{2}(\R^d)} , \quad \| u_{\e} \|_{L^2(\OO_{\e})}  \leq C \s_\e^2 \|f\|_{L^{2}(\R^d)}.$$
    Hence, $\s_\e^{-2}u_\e \to u$ weakly in $L^{2}(\R^d)$.
\end{itemize}
\end{proposition}

\medskip

The corresponding cell problem is
\be\label{pb-cell-lap}
\left\{\begin{aligned}
-\Delta w_\eta &=  c_{\eta}^{2},\ &&\mbox{in}\ Q_\eta := Q_0 \setminus (\eta T),\\
w_\eta &=0,\ &&\mbox{on} ~ \eta T,\\
 w_\eta &\  \mbox{is $Q_0$-periodic,}
\end{aligned}\right.\nn
\ee
where $c_\eta$ is the same as before, see \eqref{def-c-eta}. The solution satisfies
\ba\label{est-weta-2-lap}
\|\nabla w_\eta \|_{L^2(Q_\eta)} \leq C c_{\eta} , \quad \| w_\eta  \|_{L^2(Q_\eta)} \leq C.\nn
\ea

Then define
\be\label{w-q-e-def-lap}
w_{\eta,\e} (\cdot) : = w_{\eta} \big(\frac{\cdot}{\e}\big),\nn
\ee
which solves
\be\label{w-q-e-pt1-lap}
\left\{\begin{aligned}
- \Delta w_{\eta,\e} & = \e^{-2} c_{\eta}^{2} = \s_{\e}^{-2},\ &&\mbox{in}\ \e Q_0 \setminus (a_\e T),\\
w_{\eta,\e} &=0,\ &&\mbox{on}  \ a_\e T,\\
w_{\eta,\e}  &\  \mbox{is $\e Q_0$-periodic}.
\end{aligned}\right.\nn
\ee
Clearly $w_{\eta,\e}$ vanishes on the holes. For each $R>1$, by \eqref{est-weta-2} and the periodicity of $w_{\eta}$, direct calculation gives
\ba\label{w-q-e-pt2-lap}
\|w_{\eta,\e}\|_{L^2(B(0,R))} \leq C(R),  \quad \|\nabla w_{\eta,\e}\|_{L^2(B(0,R))}  \leq C(R) \s_{\e}^{-1},\nn
\ea
where the constant $C(R)$ depends only on $R$. Using one more time the periodicity of $w_{\eta}$ implies
\be\label{w-q-e-lim-lap}
w_{\eta,\e}  \to \bar w  \ \mbox{weakly in} \ L^{2}_{loc}(\R^d),
\ee
as $\e\to 0$, up to possible extraction of subsequences. Here $\bar w:= \int_{Q_0} w \,\dx$ where $w$ is the weak limit of $w_\eta$ in $L^2(Q_0)$.

\medskip

For each $\phi \in C_c^\infty(\R^d)$, testing \eqref{Poisson-Oe} by $\phi w_{\eta,\e}$ gives
\ba\label{Lap-O-weak}
\int_{\R^d} \nabla u_\e : \nabla( \phi w_{\eta,\e} )\,\dx  = \langle f , (w_{\eta,\e} \phi \rangle.
\ea
It is left to pass $\e\to 0$ in \eqref{Lap-O-weak}. This can be done case by case similarly as the Stokes problem and we will not repeat the details.


\end{document}